\documentclass[11pt, a4paper]{amsart}

\usepackage{amsmath}
\usepackage{amssymb}
\usepackage{graphicx}
\usepackage{color}
\usepackage{url}
\usepackage{blkarray}
\usepackage{hyperref}
\hypersetup{
    colorlinks,
    citecolor=black,
    filecolor=black,
    linkcolor=black,
    urlcolor=black
}

\newtheorem{thm}{Theorem}
\newtheorem{cor}[thm]{Corollary}
 
\newtheorem{lem}[thm]{Lemma} 
\newtheorem{rem}[thm]{Remark} 
\newtheorem{prop}[thm]{Proposition} 
\newtheorem{porism}[thm]{Porism} 
\newtheorem{example}[thm]{Example}

\begin{document}

\title{On the signature of a positive braid}

\author{Joshua Evan Greene}
\author{Livio Liechti}
\email{joshua.greene@bc.edu}
\email{livio.liechti@unifr.ch}

\begin{abstract} 
We show that the signature of a positive braid link is bounded from below by one-quarter of its first Betti number. This equates to one-half of the optimal bound conjectured by Feller, who previously provided a bound of one-eighth.
\end{abstract}

\maketitle

\section{Introduction}

The signature~$\sigma(L)$ of an oriented link~$L$ was introduced by Trotter~\cite{Trotter}.
Its definition leads to the classic lower bound~$\sigma(L) \le b_1(L)$ on the minimum first Betti 
number of any Seifert surface for~$L$.
For certain classes of links, this bound can be reversed, up to scale\footnote{While this does not 
coincide with the most commonly used definitions, we adopt Rudolph's convention that positive 
links have positive signature.}.
The first such result is due to Feller, who proved that~$\sigma(L) \ge \frac{1}{100} b_1(L)$ for a 
positive braid closure~$L$ and conjectured an optimal slope of~$\frac12$ for these links~\cite{Peschito}.
Notably, Rudolph had earlier shown that~$\sigma(L) \ge 0$, and Stoimenow had shown 
that~$\sigma(L) \ge b_1(L)^{1/3}$, for a positive braid closure $L$~\cite{Rudolph,Stoimenow}.
Subsequently, the slope was improved to~$\frac{1}{24}$ for the more general class of positive 
links by Baader, Dehornoy and the second author~\cite{BDL}, and to~$\frac18$ by Feller for 
positive braid closures~\cite{Pegascito}.
\smallskip

Our goal is to establish the following result, striking within one-half of Feller's conjectured bound:

\begin{thm}
\label{quarter_thm}
For every positive braid closure~$L$ that is not an unlink, we have
\[
\sigma(L)\ge \frac{b_1(L)}{4}+\frac{1}{2}.
\]
\end{thm}

Our proof of Theorem~\ref{quarter_thm} uses the signature formula of Gordon and Litherland for 
the Goeritz form of chessboard surfaces~\cite{GL}, as in~\cite{BDL,Pegascito}.
Our advance in the case of a positive braid closure stems from a careful choice of subspaces on which 
we are able to tightly control the signature of the Goeritz form.
\smallskip

One-half of the signature of a knot is a lower bound for the topological four-genus~\cite{KT, Murasugi}. 
Theorem~\ref{quarter_thm} immediately implies the following lower bound for the topological four-genus 
in terms of the usual Seifert genus. 

\begin{cor}
\label{4genus_cor}
The topological four-genus of a positive braid knot is greater than one-quarter of the Seifert genus. 
\end{cor}

For the consequences of Corollary~\ref{4genus_cor} with respect to concordance, 
we refer to the discussions by Stoimenow~\cite{Stoimenow} as well as Baader, Dehornoy and the 
second author~\cite{BDL}.
\smallskip

Another application of Theorem~\ref{quarter_thm} concerns the location of the zeroes of the 
Alexander polynomial of positive braid closures. Dehornoy noted that for random positive braids, 
these zeroes seem to accumulate on rather specific lines depending on the braid index and the 
probability of the braid generators, as in Figure 3.4 and the discussion surrounding it in~\cite{Dehornoy}. 
Since the signature is a lower bound for the number of zeroes of the Alexander polynomial that lie 
on the unit circle~\cite{GilmLiv, Lie16}, Theorem~\ref{quarter_thm} at least explains why there is a 
substantial number of zeroes on the unit circle. 

\begin{cor}
\label{Alex_cor}
For a positive braid link, more than one-quarter of the zeroes of the Alexander polynomial lie on 
the unit circle. 
\end{cor}

\emph{Organisation}. In Section~\ref{s: background}, we introduce the necessary background
on positive braids and their diagram combinatorics, as well as chessboard surfaces and the 
Goeritz form. In Section~\ref{s: tridiagonal}, we study the signature of specific integer matrices 
that we call trisum matrices. These matrices appear again as the 
matrices of the Goeritz form restricted to certain subspaces in Section~\ref{s: proof}, where we 
give a proof of Theorem~\ref{quarter_thm}. Finally, in Section~\ref{s: discussion}, we discuss 
limitations of our approach. 
\smallskip

\emph{Acknowledgements}. 
We thank the anonymous referees for their comments on a first version of this article. 
JEG was supported by the National Science Foundation under Award No. DMS-2304856 and 
by a Simons Fellowship.


\section{Background}
\label{s: background}

\subsection{Diagram combinatorics}
\label{ss diagrams}
Let~$\beta$ be a~\emph{positive braid} on~$n+1$ strands, that is, a product of positive powers 
of the standard braid generators~$\sigma_1, \dots, \sigma_n$ of the braid group~$B_{n+1}$.
The \emph{standard link diagram}~$D$ for the closure~$L$ of~$\beta$ contains a positive 
crossing for every occurrence of a braid generator~$\sigma_i$ in~$\beta$. The \emph{index} of 
a crossing is the index of the standard braid generator corresponding to it. An example of a 
standard diagram is shown in Figure~\ref{braid_example}. 
\smallskip

\begin{figure}[h]
\def\svgwidth{340pt}
\begingroup%
  \makeatletter%
  \providecommand\color[2][]{%
    \errmessage{(Inkscape) Color is used for the text in Inkscape, but the package 'color.sty' is not loaded}%
    \renewcommand\color[2][]{}%
  }%
  \providecommand\transparent[1]{%
    \errmessage{(Inkscape) Transparency is used (non-zero) for the text in Inkscape, but the package 'transparent.sty' is not loaded}%
    \renewcommand\transparent[1]{}%
  }%
  \providecommand\rotatebox[2]{#2}%
  \newcommand*\fsize{\dimexpr\f@size pt\relax}%
  \newcommand*\lineheight[1]{\fontsize{\fsize}{#1\fsize}\selectfont}%
  \ifx\svgwidth\undefined%
    \setlength{\unitlength}{421.49384264bp}%
    \ifx\svgscale\undefined%
      \relax%
    \else%
      \setlength{\unitlength}{\unitlength * \real{\svgscale}}%
    \fi%
  \else%
    \setlength{\unitlength}{\svgwidth}%
  \fi%
  \global\let\svgwidth\undefined%
  \global\let\svgscale\undefined%
  \makeatother%
  \begin{picture}(1,0.5698742)%
    \lineheight{1}%
    \setlength\tabcolsep{0pt}%
    \put(0,0){\includegraphics[width=\unitlength]{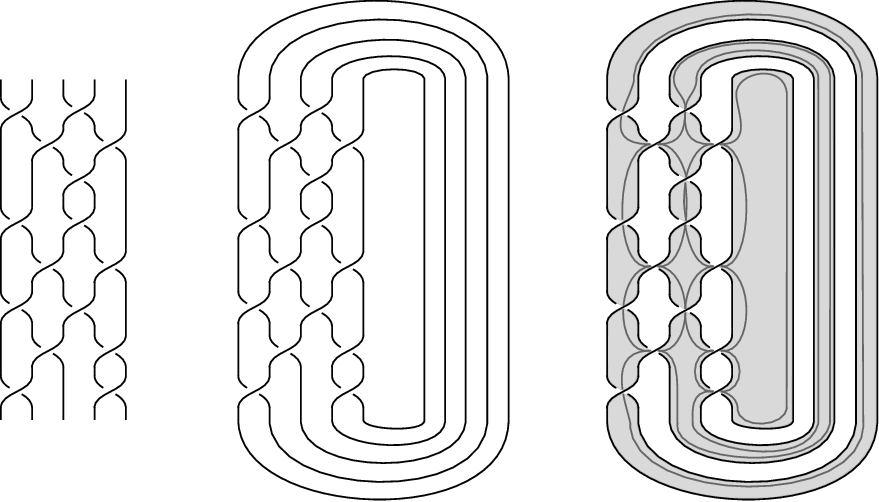}}%
    \put(0.00984995,0.06281152){\color[rgb]{0,0,0}\makebox(0,0)[lt]{\lineheight{0}\smash{\begin{tabular}[t]{l}$\sigma_1$\end{tabular}}}}%
    \put(0.04187889,0.06281152){\color[rgb]{0,0,0}\makebox(0,0)[lt]{\lineheight{0}\smash{\begin{tabular}[t]{l}$\sigma_2$\end{tabular}}}}%
    \put(0.0774666,0.06281152){\color[rgb]{0,0,0}\makebox(0,0)[lt]{\lineheight{0}\smash{\begin{tabular}[t]{l}$\sigma_3$\end{tabular}}}}%
    \put(0.1130543,0.06281152){\color[rgb]{0,0,0}\makebox(0,0)[lt]{\lineheight{0}\smash{\begin{tabular}[t]{l}$\sigma_4$\end{tabular}}}}%
    \put(0.85487648,0.28547075){\color[rgb]{0,0,0}\makebox(0,0)[lt]{\lineheight{0}\smash{\begin{tabular}[t]{l}$F$\end{tabular}}}}%
  \end{picture}%
\endgroup%

\caption{The geometric realization of the positive braid $\sigma_1\sigma_4^2\sigma_2\sigma_1
\sigma_3\sigma_2\sigma_4\sigma_1\sigma_3^2\sigma_2\sigma_4\sigma_1\sigma_3$, 
its standard diagram and a chessboard coloring thereof. On the right, the grey curves 
indicate a basis of the first homology of the black chessboard surface.}
\label{braid_example}
\end{figure}

Let~$\mathrm{cr}(D)$ be the number of crossings of~$D$. There are different types of faces 
in the diagram~$D$. For each crossing of the diagram there is a face that starts above it.
There are~$\mathrm{cr}(D)$ faces of this type. Furthermore, there is an unbounded face~$F'$ 
and a face~$F$ containing the braid axis; compare with Figure~\ref{braid_example}.  Among 
the~$\mathrm{cr}(D)$ faces that belong to the first category, let~$f_i$ be the number of faces 
with~$i$ sides, for~$i\ge2$. Locally, every crossing is met by four faces, so we have
\[
4\mathrm{cr}(D) = \sum_{i\ge2} if_i + s+s',
\]
where~$s$ and~$s'$ are the number of sides of~$F$ and~$F'$, respectively. 
\smallskip

Since 
\begin{align*}
\sum_{i\ge 2} f_i = \mathrm{cr}(D),
\end{align*}
subtracting~$4\mathrm{cr}(D)$ from both sides yields
\[
0 = \sum_{i\ge2} (i-4)f_i + s+s',
\]
and thus
\begin{align}
\label{constraint}
2f_2 + f_3 = s+ s' + \sum_{i\ge5}(i-4)f_i.
\end{align}

We will use another fact about nonsplit positive braid links~$L$ on~$n+1$ strands with standard 
diagram~$D$, namely
\begin{align}
\label{Betti}
b_1(L) = \mathrm{cr}(D)-n.
\end{align}
This is a consequence of Stallings's construction~\cite{Stallings} of a fibre surface (which must be 
genus-minimising) for nonsplit positive braid closures with first Betti number~$\mathrm{cr}(D)-n$.

\subsection{Chessboard surfaces}
Choose the chessboard colouring of the standard diagram~$D$ defined by the following property: 
the faces above odd index crossings are black, as in Figure~\ref{braid_example}. 
Let~$S_B$ be the (not necessarily orientable) surface defined by the black faces of this colouring, 
with~$\partial S_B = L$, and let~$S_W$ be the (not necessarily orientable) surface defined by the 
white faces of this colouring. 
We have
\begin{equation}
\label{e: homology sum}
\mathrm{dim}(H_1(S_B)) + \mathrm{dim}(H_1(S_W)) = \mathrm{cr}(D).
\end{equation}

\subsection{Goeritz form}
The Goeritz form~$G$ is a symmetric bilinear form on the first homology~$H_1(S)$ of a compact, 
not necessarily orientable surface~$S$ embedded in~$\mathbb{S}^3$; see Goeritz's article~\cite{Goeritz} 
or also Chapter 9 in Lickorish's book~\cite{Lickorish}. For two simple closed curves~$\gamma_1$ 
and~$\gamma_2$ in~$S$, the Goeritz form on the corresponding homology elements is defined 
as~$G(\gamma_1, \gamma_2) = \mathrm{lk}(\gamma_1,\gamma_2^\pm)$. Here,~$\mathrm{lk}$ 
denotes the linking number and~$\gamma_2^\pm$ denotes the two-sided push-off of~$\gamma_2$ 
along the normal direction to~$S$. In order to obtain positivity of the signature for positive links,
we (nonstandardly) define the linking number to count one-half of the negative crossings minus 
one-half of the positive crossings between links. 
\smallskip

If~$S$ is orientable, then by definition~$G$ equals the symmetrised Seifert form and thus only 
depends on ~$\partial S$. In particular,~$\sigma(G) = \sigma(\partial S) = \sigma(L)$, the signature 
invariant of the boundary link. If~$S$ is nonorientable, a correction term is necessary, leading to the 
formula~$\sigma(L) = \sigma(G) - \mu$ by Gordon and Litherland; see~\cite[Theorem 6]{GL} and 
the discussion following it, which extends the result to links. Here,~$\mu$ is a correction term counting the 
number of positive minus the number of negative crossings among 
all crossings of a diagram where any local orientation of the chessboard surface fails to induce the 
correct link orientation on the boundary. Our nonstandard definition of the linking number results in 
a factor~$-1$ for both~$\sigma(L)$ and~$\sigma(G)$ in the formula, so in our convention, the formula 
reads~$\sigma(L) = \sigma(G) + \mu$. Since~$\mu$ counts every crossing of a diagram for exactly 
one of the two chessboard surfaces, we sum both bounds to obtain the following result, which we use 
later on. 

\begin{thm}
\label{Gordon-Litherland}
 Let~$D$ be a positive diagram of a link~$L$, and let~$S_B$ and~$S_W$ be the two associated 
 chessboard surfaces with Goeritz forms~$G_B$ and~$G_W$, respectively. Then 
 \[
 2\sigma(L) = \sigma(G_B) + \sigma(G_W) + \mathrm{cr}(D).
\]
\end{thm}

In order to use Theorem~\ref{Gordon-Litherland}, we have to consider the Goeritz forms~$G_B$ 
and~$G_W$ of the chessboard sufaces~$S_B$ and~$S_W$, respectively. For the respective first 
homologies, we pick a basis that consists of curves~$\gamma$ winding around white and black faces 
in the counterclockwise sense: see Figure~\ref{braid_example}.

\begin{lem}
\label{Goer_calc}
Let~$\beta$ be a positive braid on~$n+1$ strands that uses each generator~$\sigma_i$ at least twice, 
and let~$D$ be the standard diagram for its closure~$L$. Let~$G_B$ and~$G_W$ be the Goeritz forms 
of the chessboard sufaces~$S_B$ and~$S_W$, respectively. We have the following:
\begin{enumerate}
\item[(i)] 
if~$\gamma$ winds around an $n$-sided face that is not~$F$ or~$F'$,~$G_\ast(\gamma,\gamma) = 4-n$;
\item[(ii)] 
if~$\gamma_1$ and~$\gamma_2$ wind around adjacent faces of the same index whose boundaries meet 
in one crossing,~$G_\ast(\gamma_1, \gamma_2) = -1$;
\item[(iii)] 
if~$\gamma_1$ and~$\gamma_2$ wind around faces with no common crossing in their 
boundaries,~$G_\ast(\gamma_1, \gamma_2) = 0$;
\item[(iv)] 
if~$\gamma_1$ and~$\gamma_2$ wind around faces whose indices are apart by two and whose boundaries 
meet in one common crossing,~$G_\ast(\gamma_1, \gamma_2) = 1$.
\end{enumerate}
\end{lem}

\begin{proof}
All four facts can be checked directly using the definition. Alternatively, one can use Lickorish's 
combinatorial description of the coefficients of a Goeritz matrix in Chapter 9 of~\cite{Lickorish}. 
Comparing with Lickorish's description, one should keep in mind that our nonstandard definition 
of the linking number results in a sign change for each coefficient. 
\end{proof}

\begin{rem}
\label{Goeritz_offdiag_remark}\emph{
We later change the orientation of certain curves~$\gamma_i$ in order to get particularly simple, 
nonnegative Goeritz matrices. 
We note that~(ii) in Lemma~\ref{Goer_calc} holds if the curves~$\gamma_1$ and~$\gamma_2$ 
are both oriented counterclockwise or if both are oriented 
clockwise. Otherwise,~$G_\ast(\gamma_1, \gamma_2) = +1$. Similarly,~(iv) in Lemma~\ref{Goer_calc} 
holds if both curves are oriented counterclockwise 
or if both are oriented clockwise.
}\end{rem}


\section{A signature bound for tridiagonal and trisum matrices}
\label{s: tridiagonal}

In this section, we discuss signature bounds for certain tridiagonal matrices. Further, we introduce and study 
a generalisation of them which we call trisum matrices. These matrices and their direct sums appear as the 
Gram matrices we study in the next section, when we restrict the Goeritz forms of the chessboard 
surfaces to appropriate subspaces.

\subsection{Tridiagonal matrices.}
Let~$T(d_1,\dots,d_n)$ denote the tridiagonal matrix with diagonal entries~$d_1,\dots,d_n \in \mathbb{Z}$ 
and with 1s on the secondary diagonals. We use the shorthand~$d^{\underline a}$ to denote a 
string of~$a$ copies of~$d$, where~$a,d \in \mathbb{Z}$,~$a \ge 0$.
\smallskip

The first lemma is easy and well-known, and we supply one short proof related to those that follow.

\begin{lem}
\label{pathsignature_lemma}
If~$\epsilon = \pm1$ and~$M = T((\epsilon \cdot 2)^{\underline a})$,
 then~$\sigma(M)= \epsilon \cdot \mathrm{dim}(M)$.
\end{lem}

\begin{proof}
The Gram-Schmidt algorithm shows that~$M$ is congruent over~$\mathbb{Q}$ to the diagonal matrix 
with entries~$\epsilon \cdot \frac{2}{1}$, $\epsilon \cdot \frac{3}{2}$, $\dots$, $\epsilon \cdot \frac{a+1}{a}$.
It follows that~$M$ is definite with sign~$\epsilon$, so~$\sigma(M) = \epsilon \cdot \dim(M)$.
\end{proof}

The following result holds for general integer coefficients on the diagonal. Let~$\underline{tr}(M)$ 
denote the sum of the negative entries on the diagonal of the matrix~$M$.

\begin{prop}
\label{diagonal_sum_prop}
If~$M = T(d_1,\dots,d_n)$, then
\[
\sigma(M) \ge -\frac{1}{2} + \frac{1}{2} \underline{tr}(M).
\]
\end{prop}

\begin{rem}
\label{diagonal_sum_remark}
\emph{
\begin{enumerate}
\item
If~$d_1,\dots,d_n \ge 0$, then~$\sigma(M) \ge 0$: for lowering diagonal entries does not raise the signature, 
and $\sigma(M)=0$ when all $d_i$ equal zero.
\item 
The matrix~$M = T(-1,-2,\dots,-2)$ is negative definite and attains the bound in Proposition~\ref{diagonal_sum_prop}.
\item
The bound in Proposition~\ref{diagonal_sum_prop} does not depend on the number or sum of positive 
diagonal coefficients. For example, if~$M = T(0,N,0,N,\dots,0,N,-1)$ with~$N$ any positive integer, 
then~$\sigma(M) = -1$, which coincides with the bound in Proposition~\ref{diagonal_sum_prop}.
\item 
Since~$\sigma(-M) = - \sigma(M)$, Proposition~\ref{diagonal_sum_prop} also implies the upper bound
\[
\sigma(M) \le \frac{1}{2} - \frac{1}{2} \underline{tr}(-M)= \frac{1}{2} + \overline{tr}(M),
\]
where~$\overline{tr}(M)$ denotes the sum of the positive diagonal coefficients of the matrix $M$. 
We shall not make use of this bound.
\end{enumerate}
}
\end{rem}

\begin{proof}[Proof of Proposition~\ref{diagonal_sum_prop}]
We proceed by induction on the dimension~$n$ of the matrix~$M$.
In the one-dimensional case~$n=1$, the statement clearly holds. We note that the 
summand~$-\frac{1}{2}$  is necessary precisely for the matrix~$(-1)$.
\smallskip

We need a second base case~$n=2$ in case the first diagonal coefficient is~$d_1=0$ or if the 
first diagonal coefficient is~$d_1 =-1$ and the second is~$d_2\ge 0$. In the first case,~$\sigma(M) = 0$, 
so the bound is satisfied. In the second case~$\sigma(M)=0$ and the bound is satisfied, even improved 
by a summand~$+\frac{1}{2}$.
\smallskip

We now assume that~$n\ge2$, and~$n\ge3$ in case~$d_1 =0$ or if~$d_1= -1$ and~$d_2 \ge 0$. 
We distinguish cases depending on the first diagonal coefficient~$d_1$. 
\smallskip

\emph{Case 1: $d_1 \ge 1$.}
The signature of the matrix
\begin{gather*}
  M =\ \begin{blockarray}{crrrrc}
    \begin{block}{(c|rrrrc)}
      d_1 & 1 & \ & \ & \ \\ \BAhhline{-------------~}
      1 & d_2 & 1 & \ & \ \\ 
      \ & 1 & \ddots &   & \\  
      \ & \ &  &   &  \\
    \end{block}
  \end{blockarray}\ 
\sim\ 
  \begin{blockarray}{crrrrc}
    \begin{block}{(c|rrrrc)}
      d_1 &  & \ & \ & \ \\ \BAhhline{-------------~}
       & d_2-\frac{1}{d_1} & 1 & \ & \ \\ 
      \ & 1 & \ddots &   & \\  
      \ & \ &  &   &  \\
    \end{block}
  \end{blockarray}\ 
\end{gather*}
is bounded from below by~$1+\sigma(M')$, where~$M'$ is the matrix obtained from~$M$ by deleting 
the first row and the first column and reducing the diagonal coefficient~$d_2$ by 1. 
We have~$\underline{tr}(M') \ge \underline{tr}(M) - 1$. By induction hypothesis, 
the statement is true for~$M'$, and we obtain
\[
\sigma(M) \ge 1 + \sigma(M') \ge 1 -\frac{1}{2}+ \frac{1}{2} \underline{tr}(M') \ge \frac{1}{2} \underline{tr}(M).
\]
We see that an occurrence of this case actually improves the signature bound by at least~$\frac{1}{2}$, 
and the signature bound is even improved by~$1$ if~$d_2 > 0$.
\smallskip

\emph{Case 2: $d_1 = 0$.}
The signature of the matrix
\begin{gather*}
  M =\ \begin{blockarray}{crrrrc}
    \begin{block}{(cr|rrrc)}
      0 & 1 & 0 & 0 & \cdots\\
      1 & d_2 & 1 & {0} & \cdots\\ \BAhhline{-------------~}
      0 & 1 &  &   & \\  
      0 & {0} &  & M'  &  \\
      \vdots & \vdots &  &  &   \\
    \end{block}
  \end{blockarray}\
\sim\
  \begin{blockarray}{crrrrc}
    \begin{block}{(cr|rrrc)}
      0 & 1 & 0 & 0 & \cdots\\
      1 & d_2 & 0 & {0} & \cdots\\ \BAhhline{-------------~}
      0 & 0 &  &   & \\  
      0 & {0} &  & M'  &  \\
      \vdots & \vdots &  &  &   \\
    \end{block}
  \end{blockarray}\
\end{gather*}
equals the signature of the matrix~$M'$. Since $\underline{tr}(M') \ge \underline{tr}(M)$, we obtain the 
desired bound for~$\sigma(M)$ by induction hypothesis. 
\smallskip

\emph{Case 3: $d_1 = -1$.}
The signature of the matrix
\begin{gather*}
  M =\ \begin{blockarray}{crrrrc}
    \begin{block}{(c|rrrrc)}
      -1 & 1 & \ & \ & \ \\ \BAhhline{-------------~}
      1 & d_2 & 1 & \ & \ \\ 
      \ & 1 & \ddots &   & \\  
      \ & \ &  &   &  \\
    \end{block}
  \end{blockarray}\ 
\sim\ 
  \begin{blockarray}{crrrrc}
    \begin{block}{(c|rrrrc)}
      -1 &  & \ & \ & \ \\ \BAhhline{-------------~}
       & d_2+1 & 1 & \ & \ \\ 
      \ & 1 & \ddots &   & \\  
      \ & \ &  &   &  \\
    \end{block}
  \end{blockarray}\ 
\end{gather*}
equals~$-1+\sigma(M')$, where~$M'$ is the matrix obtained from~$M$ by deleting the first row and the 
first column and replacing the diagonal coefficient~$d_2$ by~$d_2+1$. We distinguish two cases. 
If~$d_2<0$, then~$\underline{tr}(M') = \underline{tr}(M) + 2$. In particular, the desired lower bound 
for~$\sigma(M)$ follows from the induction hypothesis on~$M'$. If~$d_2\ge 0$, 
then~$\underline{tr}(M') = \underline{tr}(M) + 1$. This is not enough to prove the desired inequality 
by the induction hypothesis on~$M'$. However, we notice that after such a case, the matrix~$M'$ is 
as in Case~1 of our case distinctions, which in turn improves the signature bound for~$M'$ by at 
least~$\frac{1}{2}$. This exactly compensates for the loss in the present case and the desired bound is 
attained after two steps. 
\smallskip

\emph{Case 4: $d_1 \le -2$.}
The signature of the matrix
\begin{gather*}
  M =\ \begin{blockarray}{crrrrc}
    \begin{block}{(c|rrrrc)}
      d_1 & 1 & \ & \ & \ \\ \BAhhline{-------------~}
      1 & d_2 & 1 & \ & \ \\ 
      \ & 1 & \ddots &   & \\  
      \ & \ &  &   &  \\
    \end{block}
  \end{blockarray}\ 
\sim\ 
  \begin{blockarray}{crrrrc}
    \begin{block}{(c|rrrrc)}
      d_1 &  & \ & \ & \ \\ \BAhhline{-------------~}
       & d_2-\frac{1}{d_1} & 1 & \ & \ \\ 
      \ & 1 & \ddots &   & \\  
      \ & \ &  &   &  \\
    \end{block}
  \end{blockarray}\ 
\end{gather*}
is bounded from below by~$-1+\sigma(M')$, where~$M'$ is the matrix obtained from~$M$ 
by deleting the first row and the first column. We have~$\underline{tr}(M') \ge  \underline{tr}(M)+2.$ 
Again, the desired inequality follows from the induction hypothesis on~$M'$.
\end{proof}

\subsection{Trisum matrices.}
We will need an extension of the lower bound of Proposition~\ref{diagonal_sum_prop} to a slightly wider class 
of matrices. 
\smallskip

Let~$N = T(d_1,\dots, d_r)$ and let~$C_i = T(2^{\underline a_i},1,2^{\underline b_i})$,~$i=1,\dots,k$. 
Let~$\widetilde M$ be the direct sum of matrices~$\widetilde M=N\oplus C_1 \oplus \cdots \oplus C_k$. 
Furthermore, for each~$i=1,\dots,k$, let $1 \le g(i) \le r$ be an index with~$d_{g(i)} \le 0$, and let~$h(i)$ 
be the index of the column of~$\widetilde M$ that contains the diagonal coefficient~1 of the block~$C_i$.
Let~$M$ be the matrix obtained from~$\widetilde M$ by letting~$M_{g(i),h(i)} = M_{h(i),g(i)} = 1$ 
for each~$i=1,\dots,k$ and keeping all other entries the same. We call~$M$ a {\em trisum} matrix 
with {\em core}~$N$ and {\em blocks}~$C_i$. Let~$\underline{tr}(M)$ denote the sum of the negative 
diagonal coefficients of its core, and let~$b(M)$ denote the sum of the dimensions of its blocks.

\begin{prop}
\label{extended_bound_prop}
If~$M$ is a trisum matrix, then 
\[
\sigma(M) \ge -\frac{1}{2} + \frac12  \underline{tr}(M) + \frac12 b(M).
\]
\end{prop}

\begin{proof}
We proceed by induction on the number~$k$ of blocks of~$M$. 
\smallskip

If~$k=0$, then the bound coincides with the bound from Proposition~\ref{diagonal_sum_prop}. 
\smallskip

Now suppose the bound holds for matrices with~$k$ blocks~$C_i$, and let~$M$ be a matrix 
with~$k+1$ blocks~$C_i$. We make a case distinction depending on the block~$C_{k+1}$. 
\smallskip

\emph{Case 1: $C_{k+1} = T(1).$} 
In this case,~$\sigma(M) = 1 + \sigma(M')$, where~$M'$ is obtained from~$M$ by deleting the 
last row and column and replacing the diagonal coefficient~$d_{g(k+1)}$ by~$d_{g(k+1)}-1$. 
Thus, $ \underline{tr}(M') =  \underline{tr}(M) -1$ and $b(M') = b(M)-1$. Applying the induction 
hypothesis for~$M'$, we get 
\begin{align*}
\sigma(M) = 1 + \sigma(M') &\ge  1 -\frac{1}{2} + \frac{1}{2}  \underline{tr}(M') + \frac{1}{2}b(M')\\
&= -\frac{1}{2} + \frac{1}{2}  \underline{tr}(M) + \frac{1}{2}b(M).
\end{align*}
\smallskip

\emph{Case 2: $C_{k+1} = T(2^{\underline a},1,2^{\underline b})$ for~$a=0$ or~$b=0$.} 
Note that the matrix~$T(2^{\underline a},1)$ is congruent to the diagonal matrix with diagonal 
coefficients~$\frac{2}{1},\frac{3}{2},\dots,\frac{a+1}{a},\frac{1}{a+1}$. 
In particular, we get that~$\sigma(M) = a+1 + \sigma(M')$, where~$M'$ is obtained from~$M$ by deleting the 
last~$a+1$ rows and columns and replacing the diagonal coefficient~$d_{g(k+1)}$ by~$d_{g(k+1)}-(a+1)$. 
As in Case~1, applying the induction hypothesis for~$M'$ gives the desired bound. 
\smallskip

\emph{Case 3: $C_{k+1} = T(2,1,2).$} 
The matrix~$T(2,1,2)$ is congruent to the diagonal matrix with diagonal coefficients~$2,0,2$. 
It follows that~$\sigma(M) = 2 + \sigma(M')$, where~$M'$ is obtained from~$M$ by deleting 
the last three rows and columns, but also the row and column of index~$g(k+1)$. 
In order to see this, first apply the base change to~$M$ that effectuates the congruence 
of the block~$T(2,1,2)$ to the diagonal matrix with diagonal coefficients~$2,0,2$.
The row and the column of index~$h(k+1)$ now have one nonzero coefficient, 
namely~$M_{g(k+1),h(k+1)} = M_{h(k+1),g(k+1)} = 1$. As a next change of base, 
the row/column of index~$h(k+1)$ can be subtracted as often as necessary 
from other rows/columns in order for the row and the column of index~$g(k+1)$ to have only
their diagonal coefficient nonzero, as well as the coefficients~$M_{g(k+1),h(k+1)} = M_{h(k+1),g(k+1)} = 1$. 
After this change of base, the last three indices together with index~$g(k+1)$ become a direct 
summand of the matrix~$M$, with signature~2. 

We remark that deleting the row and column of index~$g(k+1)$ makes the blocks~$C_i$ with~$g(i) = g(k+1)$ 
into direct summands of the matrix~$M'$. Therefore~$M'$ decomposes as~$M'_1 \oplus M_2' \oplus C$ 
where each~$M_j'$ has at most~$k$ of the blocks~$C_i$ and is of the form so that we can apply the induction 
hypothesis, and~$C$ consists of direct summands of the type~$C_i$. For a block of type~$C_i$, 
we have~$\sigma(C_i) \ge \frac{1}{2}\dim(C_i)$. Applying the induction hypothesis to~$M_j'$, we obtain
\begin{align*}
\sigma(M) &= 2 + \sigma(M') \\ 
&= 2 + \sigma(M'_1) + \sigma(M'_2) + \sigma(C) \\
&\ge 2 -\frac{1}{2} + \frac{1}{2}\underline{tr}(M_1')+ \frac{1}{2}b(M_1') -\frac{1}{2}+ \frac{1}{2}\underline{tr}(M_2')+\frac{1}{2}b(M_2') + \frac{1}{2}\dim(C) \\
&\ge 1 + \frac{1}{2} \underline{tr}(M) + \frac{1}{2}\left(b(M_1') + b(M_2') + \dim(C)\right) \\
&= -\frac{1}{2} + \frac{1}{2} \underline{tr}(M) + b(M), 
\end{align*}
where the first inequality uses the induction hypothesis for both~$M'_1$ and~$M'_2$, 
and the second inequality uses the fact that~$d_{g(k+1)}\le 0$, which in particular means 
that~$\underline{tr}(M_1') + \underline{tr}(M_2') \ge \underline{tr}(M)$.The final equality follows 
from~$\dim(C_{k+1}) = 3$.
\smallskip

\emph{Case 4: $C_{k+1} = T(2^{\underline a},1,2^{\underline b})$ for~$a\ge1$ and~$b\ge1$.} 
Case 3 covers the case that~$a=b=1$, so we can suppose that~$a>1$ or~$b>1$. 
The matrix~$T(2^{\underline a},1,2^{\underline b})$ is congruent to the diagonal matrix with 
coefficients~$\frac{2}{1},\dots,\frac{a+1}{a},x, \frac{b+1}{b},\dots,\frac{2}{1},$
where~$x= 1-\frac{a}{a+1}-\frac{b}{b+1}<0$. In particular, $\sigma(M) \ge a+b-1 + \sigma(M')$, 
where~$M'$ is obtained from~$M$ by deleting the last~$a+b+1$ rows and columns. 
Applying the induction hypothesis for~$M'$ yields the desired bound, since~$a+b\ge 3$
and therefore the improvement~$a+b-1$ is at least half as large as the deficit~$a+b+1$ for the 
bound we have when we apply the induction hypothesis to~$M'$.
\end{proof}

The congruences noted in the proof of Proposition \ref{extended_bound_prop} establish the following 
useful result.

\begin{porism}
\label{signature_porism}
If $M = T(2^{\underline a},1,2^{\underline b})$, then
\[
\sigma(M) =
\begin{cases}
\dim(M), &\mathrm{if} \min\{a,b\}=0, \\
\dim(M)-1, & \textup{if } a=b=1, \textup{ and}\\
\dim(M)-2, & \mathrm{otherwise}.
\end{cases}
\]
In particular, $\sigma(M) \ge \frac12 \dim(M)$. \qed
\end{porism}


\section{Proof of Theorem~\ref{quarter_thm}} 
\label{s: proof}

\subsection{Preliminaries.}
Let~$L$ be a positive braid link, and let~$\beta$ be a positive braid whose closure is~$L$ such that
\begin{enumerate}
\item[(a)]
$\beta$ has minimal index~$n \ge 1$ among all positive braids whose closure is~$L$ and
\item[(b)]
the sum of the indices of the generators appearing in~$\beta$ is minimal subject to (a).
\end{enumerate}

We quickly dispense with some trivial cases:

\begin{itemize}
\item
If~$n=1$, then~$L$ is the unknot, and~$\sigma(L) = b_1(L) = 0$.
\item
If~$n \ge 2$ and~$\beta$ does not use the generator~$\sigma_i$ for some~$1 \le i \le n-1$, 
then~$L$ is the split union~$L_1 \sqcup L_2$, where each of~$L_1$ and~$L_2$ is a positive braid link.
In this case,~$\sigma(L) = \sigma(L_1) + \sigma(L_2)$ and~$b_1(L) = b_1(L_1) + b_1(L_2)$.
\item
If~$n \ge 2$ and~$\beta$ uses a single generator~$\sigma_i$ for some~$1 \le i \le n-1$, then~$L$ is the 
connected sum~$L_1 \# L_2$, where each of~$L_1$ and~$L_2$ is a positive braid link. 
Once more, in this case, we have~$\sigma(L) = \sigma(L_1) + \sigma(L_2)$ 
and~$b_1(L) = b_1(L_1) + b_1(L_2)$.
\item
If~$n=2$ and~$\beta$ uses~$\sigma_1$ at least twice, then~$L$ is a~$(2,k)$-torus link,~$k \ge 2$, 
and~$k-1 = \sigma(L) = b_1(L) \ge \frac14 b_1(L) + \frac12$.
\end{itemize}
Thus, Theorem~\ref{quarter_thm} follows if we can prove it under the additional assumptions that
\begin{enumerate}
\item[(c)]
$n \ge 3$ and
\item[(d)]
$\beta$ uses each generator~$\sigma_i$ at least twice,
\end{enumerate}
which we assume in the remainder of the section. In particular, we have~$s+s' \ge 4$ in the notation of 
Section~\ref{ss diagrams}, so Equation~(\ref{constraint}) entails
\begin{align}
\label{constraint2}
2f_2 + f_3 \ge 4 + \sum_{i\ge5}(i-4)f_i.
\end{align} 

\begin{lem}
\label{face_lemma}
In every column of $\beta$ except possibly the first, there exists a face with at least four sides.
\end{lem}

\begin{proof}
Let~$i > 1$ denote the index of a column of~$\beta$. Select an occurrence of~$\sigma_{i-1}$ in~$\beta$ 
and let~$f$ denote the face in the~$i$-th column incident to it. By assumption~(d), the top and bottom 
crossings incident with~$f$ exist and are distinct from one another. Suppose for a contradiction that~$f$ 
is not incident with any other crossings. Then a cyclic permutation and a sequence of distant braid moves 
converts~$\beta$ into a minimal index positive braid representation~$\beta'$ of~$L$ with the same index 
sum as~$\beta$ and which contains the subword~$\sigma_{i}\sigma_{i-1}\sigma_{i}$.
Then the braid move~$\sigma_{i}\sigma_{i-1}\sigma_{i}\to\sigma_{i-1}\sigma_{i}\sigma_{i-1}$ 
converts~$\beta'$ into a minimal index positive braid representation~$\beta''$ of~$L$ with lower index 
sum than~$\beta'$, violating assumption~(b). Hence~$f$ is incident with at least one additional crossing, 
so it has at least four sides.
\end{proof}

\subsection{Subsets and subspaces.}
In this section, we describe our choice of subspaces on which to examine the Goeritz form in 
the case of a positive braid closure.
The subspaces are chosen so that the restrictions of the Goeritz forms to them are presented by trisum matrices.
Then we can estimate their signatures using the results of Section \ref{s: tridiagonal} towards the end of proving Theorem \ref{quarter_thm}.

\subsubsection{The black surface}
We define two subsets~$B_0, B_2 \subset H_1(S_B)$ of the first homology of the black chessboard surface, 
obtained by selecting homology classes of certain curves~$\gamma$ winding around white faces. 
\smallskip

Let $B_0$ denote the set of homology classes of the following curves in~$S_B$. First, we take all 
curves~$\gamma$ corresponding to white faces above crossings of index~$i\equiv0\,(\mathrm{mod}~4)$, 
with the following exception: for each index, we omit one curve corresponding to a face with the most sides.
Note that this number of sides is at least four, by Lemma~\ref{face_lemma}. Second, we take all the curves 
corresponding to white faces with two sides above crossings of index~$i\equiv2\,(\mathrm{mod}~4)$.
Third, we take some curves corresponding to white faces with three sides above crossings of 
index~$i\equiv2\,(\mathrm{mod}~4)$. More precisely, we identify a face with at least four sides in this column.
Proceeding upwards from it, we take every other three-sided face we encounter, beginning with the first one.
In this way, we collect at least half of the white faces with three sides in each column of 
index~$\equiv2\,(\mathrm{mod}~4)$.

Similarly, define~$B_2$ just like~$B_0$, but with the role of the indices~$i\equiv0\,(\mathrm{mod}~4)$ 
and~$i\equiv2\,(\mathrm{mod}~4)$ interchanged: see Figure~\ref{B2example} for an example.

\begin{figure}[h]
\def\svgwidth{300pt}
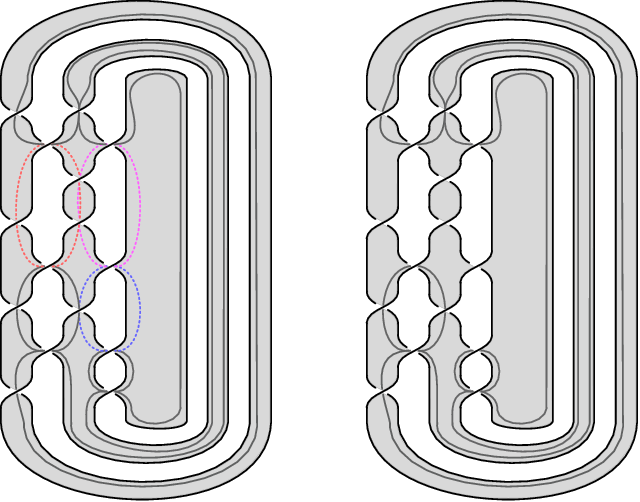
\caption{An example of selecting~$B_2 \subset H_1(S_B)$: the red dashed curve is omitted since it runs 
around a white face with the most sides among all faces of index two. The purple dashed curve is omitted 
since its corresponding face has four crossings. Finally, the blue dashed curve is omitted since it is the 
second curve running around a face with three sides when counting upwards from the purple curve.}
\label{B2example}
\end{figure}

\subsubsection{The white surface}
Now, we define two subsets~$B_1, B_3 \subset H_1(S_W)$ of the first homology of the white chessboard 
surface~$S_W$, as follows.
\smallskip 

Let~$B_1$ denote the set of homology classes of the following curves in~$S_W$. First, we take all 
curves~$\gamma$ corresponding to black faces above crossings of index~$i\equiv1\,(\mathrm{mod}~4)$, 
with the following exception: for each index, we omit one curve corresponding to a face with the most sides.
Once more, Lemma~\ref{face_lemma} implies that this face must have at least four sides, unless~$i=1$ 
and every face  in the first column has three or fewer sides; then assumption~(d) applies to show that this 
face has three sides. Furthermore, we add the curves corresponding to black faces with two sides 
above crossings of index~$i\equiv3\,(\mathrm{mod}~4)$. Finally, we add every other curve corresponding to 
white faces with three sides above crossings of index~$i\equiv3\,(\mathrm{mod}~4)$, exactly as in the 
definition of~$B_0$. 
\smallskip

Finally, define~$B_3$ just like~$B_1$, but with the role of the indices~$i\equiv1\,(\mathrm{mod}~4)$ 
and~$i\equiv3\,(\mathrm{mod}~4)$ interchanged. The exception for the leftmost column is now the following: 
when we add every other face with three sides, we round down in case the number of such faces is odd and 
there is no face with more than three sides. 

\subsubsection{Subspaces and Gram matrices.}
Define~$X_j$ to be the subspace of homology generated by~$B_j$,~$j=0,1,2,3$. Note that~$B_j$ is a basis 
of~$X_j$ for each~$j$. Define~$G_j$ to be the Gram matrix of~$B_j$ with respect to the relevant Goeritz 
form (i.e.~$G_B$, if~$j$ is even, and~$G_W$, if~$j$ is odd).
\smallskip

The first lemma in this section is more or less immediate from the construction.

\begin{lem}
\label{form_lem}
For each~$j=0,1,2,3$, the matrix~$G_j$ is a direct sum of tridiagonal matrices~$T(2^{\underline a})$, 
tridiagonal matrices~$T(2^{\underline a},1,2^{\underline b})$, and trisum 
matrices~$M_i$,~$1 \le i \le n$,~$i \equiv j \pmod 4$.
\end{lem}

\begin{proof} 
Fix a value $j=0,1,2,3$ and an index $i \equiv j \pmod 4$, $1 \le i \le n$.
As in the proof of Lemma~\ref{face_lemma}, assumption~(b) implies that there is no appearance 
of $\sigma_{i}\sigma_{i-1}\sigma_{i}$ in $\beta$, even after distant braid moves.
It follows that a 3-sided face in column~$i-2$ shares two crossings with faces in the same column 
and one crossing with a face in column~$i$; hence no 3-sided face in column~$i+2$ shares a 
crossing with a face in column~$i$.This shows that the subspace of homology generated by the 
curves we selected for the base~$B_j$ from column~$i-2$ and~$i$ gives rise to a direct summand of~$G_j$.
We now show that this direct summand is in turn a direct sum of tridiagonal matrices~$T(2^{\underline a})$, 
tridiagonal matrices~$T(2^{\underline a},1,2^{\underline b})$, and a trisum matrix~$M_i$.
\smallskip

The core of the trisum matrix~$M_i$ corresponds to the faces selected for~$B_j$ from the~$i$-th column. 
It is a tridiagonal matrix by Lemma~\ref{Goer_calc}. In fact, Lemma~\ref{Goer_calc} gives off-diagonal 
coefficients~$-1$, but following Remark~\ref{Goeritz_offdiag_remark} we may change the orientation of every 
other curve~$\gamma$ winding around a face from the~$i$-th column in order to have all 
off-diagonal coefficients~$1$. Further, the trisum matrix~$M_i$ has one block for each 3-sided face in 
column~$i - 2$ which shares a crossing with one of the selected faces in the~$i$-th column. 
This block is of the form~$T(2^{\underline a},1,2^{\underline b})$, where~$a,b\ge0$. 
This follows Lemma~\ref{Goer_calc} as well as the care we took in selecting every other 3-sided face in 
column~$i - 2$: at most one 3-sided face can appear in any chain of faces\footnote{By a chain of faces we 
mean a linearly ordered set of faces in the same column such that two faces share a crossing exactly if they 
are adjacent in the linear order.} 
we add from the column~$i-2$ during the construction of the base~$B_j$. The off-diagonal coefficient~$1$ we need 
for each block of the trisum matrix is given by~(iv) in Lemma~\ref{Goer_calc}. Again, we might need to 
invoke Remark~\ref{Goeritz_offdiag_remark} and adjust the orientations of the curves winding around the 
faces in column $i-2$ in order to have all off-diagonal coefficients~1.
\smallskip

A 3-sided face in column~$i-2$ might share its crossing with the face from the~$i$-th column we did not select 
in the construction of~$B_j$. Such a 3-sided face is responsible for a direct 
summand~$T(2^{\underline a},1,2^{\underline b})$. Finally, a chain of 2-sided faces in column~$i-2$ 
we added during the construction of the base~$B_j$ is responsible for a direct summand~$T(2^{\underline a})$. 
One more time, we might need to invoke Remark~\ref{Goeritz_offdiag_remark} and adjust the orientations 
of the curves winding around faces in column $i-2$ for off-diagonal coefficients 1.
\end{proof}

The second lemma follows directly from counting the sizes of the bases~$B_j$.

\begin{lem}
\label{dimension_lem}
\[
\sum_{j=1}^4 \mathrm{dim}(X_j) \ge f_2 + \frac{f_3-1}{2} + \mathrm{cr}(D) - n.
\]
\end{lem}

\begin{proof}
For each index~$i=1,\dots,n$, every 2-sided face in the~$i$-th column contributes a basis element 
to~$B_{j+2}$, where~$0 \le j+2 \le 3$ and~$i+2 \equiv j+2 \pmod 4$. The total number of these elements, 
summed over~$i=1,\dots,n$, is~$f_2$. At least half of the 3-sided faces in the~$i$-th column contribute a 
basis element to~$B_{j+2}$, with the exception of~$i=1$, where at least one less than half of them contribute.
The total number of these elements is thus~$\ge (f_3-1)/2$. Lastly, all but one face in the~$i$-th column 
contributes a basis element to~$B_j$, for each index~$0 \le j \le 3$ and~$i \equiv j \pmod 4$.
The total number of these elements is $\mathrm{cr}(D)-n$. The stated inequality now issues directly.
\end{proof}

The final lemma concerns the signature of the matrices~$G_j$.

\begin{lem}
\label{signature_lem}
We have
\[
\sum_{j=0}^3 \sigma(G_j) \ge -\frac{1}{2}(f_2 + f_3) + 2.
\]
\end{lem}

\begin{proof}
By Lemma \ref{form_lem},
\begin{equation}
\label{e: g sum}
\bigoplus_{j=0}^3 G_j = \left( \bigoplus_{i=1}^n M_i \right) \oplus T,
\end{equation}
where~$T$ is the direct sum of some tridiagonal matrices of the form~$T(2^{\underline a})$, $a \ge 1$, 
and~$T(2^{\underline a}, 1, 2^{\underline b})$,~$a,b \ge 0$. 
By Lemma \ref{pathsignature_lemma} and Porism~\ref{signature_porism}, we have
\begin{equation}
\label{e: T bound}
\sigma(T) \ge \frac12 \dim(T).
\end{equation}

We also have
\begin{equation}
\label{e: k bound}
\sum_{i=1}^n b(M_i) + \dim(T) = f_2 + f_3',
\end{equation}
where~$f_3'$ denotes the number of 3-sided faces used in the construction of the bases~$B_j$.
\smallskip

Let~$p$ denote the number of indices~$1 \le i \le n$ such that~$\underline{tr}(M_i) < 0$.
The negative diagonal coefficients in the core of the matrix~$M_i$ are the values of the form~$4-s$, 
where~$s \ge 5$ is the number of sides in a face in the~$i$-th column. Since we omit at least one face 
with~$\ge 5$ sides in forming the core of~$M_i$ for each of~$p$ indices, it follows that
\[
\sum_{i=1}^n \underline{tr}(M_i) \ge p+ \sum_{i\ge5} (4-i)f_i.
\]
Invoking Inequality~(\ref{constraint2}) leads to the bound
\begin{equation}
\label{e: d bound}
\sum_{i=1}^n \underline{tr}(M_i) \ge p - (2f_2 + f_3 - 4).
\end{equation}

Let~$q$ denote the number of indices~$1 \le i \le n$ such that~$\underline{tr}(M_i) = b(M_i) = 0$.
For each such index, (1) of Remark~\ref{diagonal_sum_remark} gives
\[
\sigma(M_i) \ge 0 = \frac12 \underline{tr}(M_i) + \frac12 b(M_i).
\]
For the other~$n-q$ indices, we bound~$\sigma(M_i)$ from below by Proposition \ref{extended_bound_prop}.
Summing all of these bounds, we obtain
\begin{equation}
\label{e: sigma bound}
\sum_{i=1}^n \sigma(M_i) \ge -\frac{n-q}{2} + \frac12 \sum_{i=1}^n \underline{tr}(M_i) 
+ \frac12 \sum_{i=1}^n b(M_i).
\end{equation}

It follows that the number of indices~$1 \le i \le n$ for which~$\underline{tr}(M_i) = 0$ and~$b(M_i) \ge 1$ 
is equal to~$n-p-q$. Thus,~$f_3' \ge n-p-q$, which rearranges to
\begin{equation}
\label{e: f3' bound}
p+q -n + f_3' \ge 0.
\end{equation}

Consequently,
\begin{eqnarray*}
\sum_{j=0}^3 \sigma(G_j) &\overset{\eqref{e: g sum}}{=}& \sum_{i=1}^n \sigma(M_i) + \sigma(T) \\
&\overset{\eqref{e: T bound} + \eqref{e: sigma bound}}{\ge}& \frac{q-n}{2} + 
	\frac12 \sum_{i=1}^n \underline{tr}(M_i) + \frac12 \sum_{i=1}^n b(M_i) + \frac12 \dim(T) \\
&\overset{\eqref{e: k bound} + \eqref{e: d bound}}{\ge}& \frac{q-n}{2} + 
	\frac12 (p - (2f_2 + f_3 - 4))+ \frac12 (f_2+f_3') \\
&\overset{\eqref{e: f3' bound}}{\ge} & - \frac12 (f_2 +f_3) +2,
\end{eqnarray*}
as desired.
\end{proof}

\subsection{The combined signature bound}
We are nearly ready to put everything together to prove Theorem~\ref{quarter_thm}.
The last fact we shall use is that if~$G : H \times H \to \mathbb{Z}$ is a symmetric, bilinear form, 
and~$X \subset H$ is a subspace, then
\begin{equation}
\label{e: subspace bound}
\sigma(G) + \dim H \ge \sigma(G |_{X \times X}) + \dim X.
\end{equation}
That is because the left side equals two times the dimension of the largest subspace of~$H$ on 
which~$G$ is positive definite, while the right side equals two times the dimension of the largest 
subspace of~$X \subset H$ on which~$G$ is positive definite.

\begin{proof}[Proof of Theorem \ref{quarter_thm}]
As argued at the outset of the section, we may assume that hypotheses~(a)-(d) hold for~$L$, 
so the results of this section pertain to it. Using {Theorem}~\ref{Gordon-Litherland} twice~$(\ast)$, 
then~\eqref{e: subspace bound} four times~$(\ast\ast)$ and finally Lemma~\ref{signature_lem}, 
Lemma~\ref{dimension_lem} and \eqref{e: homology sum} $(\ast\ast\ast)$, 
we obtain
\begin{eqnarray*}
4\sigma(L) - 2\mathrm{cr}(D) &\overset{(\ast)}=&  2\sigma(G_B) + 2\sigma(G_W) \\
 &\overset{(\ast\ast)}\ge& \sum_{j=0}^3 \left(\sigma(G_j) + \dim(X_j)\right) - 
 	2 \left(\mathrm{dim}(H_1(S_B)) + \mathrm{dim}(H_1(S_W)\right)\\
 &\overset{(\ast\ast\ast)}\ge& \left( -\frac{1}{2}(f_2 + f_3) + 2 \right) + \left( f_2 + \frac{f_3-1}{2} + 
 	\mathrm{cr}(D) - n \right) - 2 \mathrm{cr}(D) \\
 &=& \frac{3}{2} + \frac{f_2}{2} -\mathrm{cr}(D)-n.
\end{eqnarray*}
Thus,
 \begin{align}
 \label{final_eq}
 4\sigma(L) \ge \frac{3}{2} + \frac{f_2}{2} + \mathrm{cr}(D)-n \ge \frac{3}{2} + b_1(L), 
 \end{align}
using \eqref{Betti} in the second inequality.
Both values $\sigma(L)$ and $b_1(L)$ are integers. 
Thus, we obtain the improvement $4 \sigma(L) \ge 2 + b_1(L)$, and dividing through by 4 gives 
the bound in Theorem~\ref{quarter_thm}.
\end{proof}


\section{Discussion}
\label{s: discussion}

The proof of Theorem~\ref{quarter_thm} suggests room for improvement in the direction of Feller's 
conjecture.
However, we could not quickly see how to do more, and examples suggest the need for more ideas.
On the one hand, Proposition~\ref{diagonal_sum_prop} is clearly suboptimal if, say,~$\underline{tr}(M) < -2n$, 
which suggests improving it towards the end of improving Theorem~\ref{quarter_thm}.
On the other hand, it is easy to give examples for which the relevant tridiagonal matrices 
have only $-1$s and $-2$s on the diagonal.
Then the bound in Proposition~\ref{diagonal_sum_prop} seems to be sharp, and it correspondingly 
seems hard to improve on Theorem~\ref{quarter_thm} for these examples.
In another direction, Equation~\eqref{final_eq} shows that the bound of Theorem~\ref{quarter_thm} can be 
improved in case the proportion of faces with two sides to the total number of faces is large.
However, this proportion can be arbitrarily small, as witnessed by the following example.

\begin{example}
\label{foursides_example}
\emph{
Consider the positive braid~$\beta = (\sigma_1 \cdots \sigma_n\sigma_n \cdots \sigma_1)^2 \in B_{n+1}$. 
Even up to conjugation, no braid moves can be applied to this braid. Furthermore, its closure~$L$ 
is a nontrivial and nonsplit link. Finally~$\beta$ is of minimal index representing~$L$, since the 
link~$L$ has~$n+1$ components, so a braid representing it must have at least~$n+1$ strands. 
Independently of~$n$, its standard diagram~$D$ has~$f_2 = 4$,~$f_4 = \mathrm{cr}(D)-4$, 
and~$f_k = 0$ for~$k \ne 2,4$. The signature of such links~$L$ was computed by the second 
author in Remark 15 of~\cite{Lie20}. It equals approximately two-thirds of the first Betti number. 
More precisely,~$\sigma(L) = 2n+1$ and~$\mathrm{null}(L) = n-1$.
}
\end{example}

It seems that in order to obtain a bound for the signature of positive braids via Goeritz forms that is 
better than the quarter from Theorem~\ref{quarter_thm}, one would have to deal with subspaces 
of the first homology groups of the chessboard surfaces that contain more faces with four or more sides. 
In our proof, the specific almost tridiagonal form of the matrices for the Goeritz forms restricted to our 
subspaces~$X_j$ played a crucial role. Since faces with four sides can share crossings with faces to 
both the right \emph{and} to the left, a division of the matrix for the Goeritz form into blocks that correspond 
to columns two indices apart as in Lemma~\ref{form_lem} should be impossible in general.  Perhaps more 
complicated partitions or also a new and more global approach would be necessary for further 
improvements.

\bibliographystyle{siam}

\begin{thebibliography}{99}
 \bibitem{BDL} S.\ Baader, P.\ Dehornoy and L.\ Liechti: \emph{Signature and concordance of positive knots}, Bull.\ London Math.\ Soc.~\textbf{50} (2018), no.~1, 166--173. 
 \bibitem{Peschito} P.\ Feller: \emph{The signature of positive braids is linearly bounded by their genus}, {Internat.\ J.\ Math.}~{\bf 26} (2015), no.~10, 1550081.
 \bibitem{Pegascito}{P.\ Feller}: \emph{A sharp signature bound for positive four-braids}, Q.\ J.\ Math.~\textbf{69} (2018), no.~1, 271--283.
 \bibitem{Dehornoy} P.\ Dehornoy: \emph{On the zeroes of the Alexander polynomial of a Lorenz knot}, Ann.\ Inst.\ Fourier~\textbf{65} (2015), no.~2, 509--548. 
\bibitem{GilmLiv} P.\ Gilmer, C.\ Livingston: \emph{Signature jumps and Alexander polynomials for links}, Proc.\ Am.\ Math.\ Soc.~\textbf{144} (2016), no.~12, 5407--5417. 
 \bibitem{Goeritz}{L.\ Goeritz}: \emph{Knoten und quadratische Formen}, {Math.\ Z.~}{\bf 36} (1933), 647--654.
 \bibitem{GL}{C.\ McA.\ Gordon and R.\ A.\ Litherland}: \emph{On the Signature of a Link}, {Invent.\ Math.~}{\bf 47} (1978), 53--69.
 \bibitem{KT}{L.\ H.\ Kauffman and L.\ R.\ Taylor}: \emph{Signature of links}, {Trans.\ Amer.\ Math.\ Soc.}~{\bf 216} (1976), 351--365.
\bibitem{Lie16} L.\ Liechti: \emph{Signature, positive Hopf plumbing and the Coxeter transformation (With appendix by Peter Feller and Livio Liechti)}, Osaka J.\ Math.~\textbf{53} (2016), no.~1, 251--266.
\bibitem{Lie20} L.\ Liechti: \emph{On the genus defect of positive braid knots}, Algebr.\ Geom.\ Topol.~\textbf{20} (2020), no.~1, 403--428.
\bibitem{Lickorish} W.\ B.\ Raymond Lickorish, \emph{An introduction to knot theory}, Graduate Texts in Mathematics~175, Springer, 
New York, NY, 1997.
 \bibitem{Murasugi}{K.\ Murasugi}: \emph{On a certain numerical invariant of link types}, {Trans.\ Amer.\ Math.\ Soc.}~{\bf 117} (1965), 387--422.
 \bibitem{Rudolph}{L.\ Rudolph}: \emph{Nontrivial positive braids have positive signature}, {Topology}~{\bf 21} (1982), 325--327.
\bibitem{Stallings} J.~Stallings: \emph{Constructions of fibred knots and links}, Algebraic and Geometric Topology, 55--60, Proc.\ Sympos.\ Pure Math.~\textbf{32} (1978) Amer.\ Math.\ Soc., Providence, R.I.
 \bibitem{Stoimenow}{A.\ Stoimenow}: \emph{Bennequin's inequality and the positivity of the signature}, {Trans.\ Amer.\ Math.\ Soc.}~\textbf{360} (2008), no.~10, 5173--5199. 
 \bibitem{Trotter}{H.\ F.\ Trotter}: {\it Homology of group systems with applications to knot theory}, {Ann.\ of Math.~(2)} {\bf 76} (1962), 464--498.
\end{thebibliography}

\end{document}